\documentclass{article}
\usepackage{hyperref}
\usepackage{amsmath}
\usepackage{amsthm}
\usepackage{amssymb}
\usepackage{indentfirst}
\usepackage{amsthm} 
\newtheorem{theorem}{Theorem}[section]

\newtheorem{lemma}{Lemma}[section]

\newtheorem{proposition}{Proposition}[section]

\begin{document}

\title{On complete hypersurfaces with constant scalar curvature $n(n-1)$ in the unit sphere}

\author{Jinchuan Bai, Yong Luo\footnote{corresponding author, email address: yongluo-math@cqut.edu.cn}}

\date{}

\maketitle              


\begin{abstract}
	
	Let $M^n$ be an $n$-dimensional complete and locally conformally flat hypersurface in the unit sphere $\mathbb{S}^{n+1}$ with constant scalar curvature $n(n-1)$.  We show that if the total curvature $\left( \int _ { M } | H | ^ { n } d v \right) ^ { \frac { 1 } { n } }$ of $M$ is sufficiently small, then $M^n$ is  totally geodesic.
	
\end{abstract}	
\section{Introduction}

Inspired by the famous Bernstein theorem that a minimal entire graph $M^n(n \leq7)$ immersed in $\mathbb{R}^{n+1}$ must be a hyperplane \cite{Sim}, there are many results characterizing the hyperplane as the unique minimal hypersurface in a Euclidean space under certain geometrical or analytical assumptions \cite{CL}\cite{DP}\cite{DP2}\cite{FS}\cite{SZ}\cite{SSY}. In particular, Ni \cite{Ni} and Yun \cite{Yun} proved that a complete minimal hypersurface in a Euclidean space is a hyperplane if its total scalar curvature is sufficiently small.

A hypersurface is minimal if its mean curvature is zero and the mean curvature is, in our definition, $\frac{1}{n}$ of the first elementary symmetric function of the second fundamental form. Except mean curvature, scalar curvature is the most important curvature for hypersurfaces and for hypersurfaces in a Euclidean space scalar curvature is the second elementary symmetric function of the second fundamental form.

Recently, Li, Xu and Zhou studied complete hypersurfaces in a Euclidean space with zero scalar curvature and they obtained a Ni and Yun's type Gap theorem.
 \begin{theorem}[\cite{LXZ}]\label{LXZ}
Let $M^n(n \geq 3)$ be a locally conformally flat complete hypersurface in $\mathbb{R}^{n+1}$ with zero scalar curvature. Then there exists a positive constant $C(n)$ depending only on $n$ such that $M$ is a hyperplane if  $$\int_{M}|H|^{n}dv<C(n).$$
 \end{theorem}

It is very natural to ask if one could obtain similar result for hypersurfaces in the unit sphere with zero second elementary symmetric function of the second
fundamental form. In this paper we answer this question affirmatively.
\begin{theorem}
Let $M^n(n \geq 3)$ be a complete locally conformally flat hypersurface in unit sphere  $\mathbb{S}^{n+1}$
with constant scalar curvature $n(n-1)$. There exists a sufficiently small number $\alpha$ which depends only on dimension $n$ such that if
$$\left( \int _ { M } | H | ^ { n } d v \right) ^ { \frac { 1 } { n } } < \alpha,$$
then $M$ is totally geodesic.
\end{theorem}
Theorem 1.2 is proved by analyzing the following Simons' type inequality (where to derive this inequality the assumption of locally conformally flatness is used)
$$| r | \Delta | r |\geq \frac { n } { n - 2 }\rm tr(r^3) + n | r | ^ { 2 }, $$
where $|r|^{2}=\sum_{i, j}r_{i j}^{2}$ and $r_{i j}:=R_{ij}-(n-1)\delta_{ij}$. If $M$ is complete with sufficiently small total curvature, from the above inequality we obtain that $|r|=0$, which implies by the Gauss equations that $nHh_{ij}=\sum_kh_{ik}h_{kj}$ for any $i,j=1,...,n.$ Then we can prove that $M$ has at least $n-1$ zero principle curvatures. At last we prove the following fact to finish the proof of Theorem 1.2.
\begin{proposition}
Assume that $M$ is a complete hypersurface in unit sphere  $\mathbb{S}^{n+1}$. If $M$ has at least $n-1$ zero principle curvatures everywhere, then $M$ is totally geodesic.
\end{proposition}
 Proposition 1.1 was pointed out by Wu that it can be proved by similar argument with that used in the proof of Theorem 2.2 in \cite{Wu}. For convenience of the reader we will give the details of the proof of Proposition 1.1.

The rest of the article is arranged as follows. In section 2 we will list and prove several useful Lemmas which will be used in the rest of our paper. Theorem 1.2 and Proposition 1.1 are proved in section 3.

\section{Notations and Lemmas}

    Assume that $M^n$ is a hypersurface in unit sphere $\mathbb{S}^{n+1}$. We choose a local orthonormal frame field $\left\{e_1,...,e_n, e_{n+1}\right\}$ along $M$, where $\left\{e_{i}\right\}_{i=1, \ldots, n}$ are tangent to $M$ and $e_{n+1}$ are normal to $M$. Let $\left\{\omega_{A}\right\}_{A=1, \ldots, n+1}$ be the corresponding dual coframe, and $\left\{\omega_{A B}\right\}$ the connection 1-forms on $\mathbb{S}^{n+1}$.  We make the convention on the range of indices that $1\leq A,B,...\leq n+1, 1\leq i,j,...\leq n$. The structure equation of $\mathbb{S}^{n+1}$ are
    $$
    \begin{aligned}
d\omega_A&=-\sum_{B}\omega_{A B}\wedge \omega_B, \omega_{A B}+\omega_{B A}=0,
\\ d\omega_{A B}&=-\sum_C\omega_{A C}\wedge\omega_{C B}+\frac12\sum_{C,D}K_{A B C D}\omega_C\wedge\omega_D,\\
K_{A B C D}&=\delta_{AC}\delta_{BD}-\delta_{AD}\delta_{BC},
    \end{aligned}
$$

\noindent where $K_{ABCD}$ is the curvature tensor of $\mathbb{S}^{n+1}$. When restricted on $M$ we have $\omega_{n+1}=0$. Hence $0=d\omega_{n+1}=-\sum_i\omega_{n+1i}\wedge\omega_i$. By Cartan's lemma, there exists local functions $h_{ij}$ such that
 $$
 \begin{aligned}
 \omega_{n+1i}=\sum_jh_{ij}\omega_j, h_{ij}=h_{ji}.
 \end{aligned}
 $$
 \noindent The second fundamental form is $h=\sum_{ij}h_{ij}\omega_i\otimes\omega_j.$ We also write $h=\big(h_{ij}\big)$ as a matrix and call the eigenvalues of $h$ the principle curvatures of $M$. The mean curvature of $M$ in $\mathbb{S}^{n+1}$ is defined by

$$H:=\frac{1}{n} \operatorname{tr} h.$$
The structure equations of $M$ are
$$
    \begin{aligned}
d\omega_i&=-\sum_{j}\omega_{ij}\wedge \omega_j, \omega_{ij}+\omega_{ji}=0,
\\ d\omega_{ij}&=-\sum_k\omega_{ik}\wedge\omega_{kj}+\frac12\sum_{kl}R_{ijkl}\omega_k\wedge\omega_l,
    \end{aligned}
$$
    \noindent and the Gauss equations and Codazzi equations are given by

    $$R_{i j k l}=\delta_{i k} \delta_{j l}-\delta_{i l} \delta_{j k}+h_{i k} h_{j l}-h_{i l} h_{j k}\eqno (2.1)$$

    $$ h_{i j, k}=h_{i k, j}\eqno (2.2)$$

    \noindent where the covariant derivative of $h_{i j}$ is defined by

    $$\sum_{k=1}^{n} h_{i j, k} \omega_{k}=d h_{i j}-\sum_{k=1}^{n} h_{k j} \omega_{k i}-\sum_{k=1}^{n} h_{i k} \omega_{k j}$$

    \noindent In particular,  components of the Ricci tensor are given by

    $$R_{i j}=(n-1) \delta_{i j}+n Hh_{i j}-\sum_{ k}h_{i k} h_{k j}\eqno (2.3)$$

    \noindent Assume that $R$ is the scalar curvature of $M$, then by $(2.3)$ we have

    $$R=n(n-1)+n^{2} H^{2}-|h|^{2}\eqno (2.4)$$

    \noindent where $|h|^{2}=\sum_{i, j}\left(h_{i j}\right)^{2}$ is the squared norm of the second fundamental form.

We have
\begin{lemma}
	Let $M^n (n\geq3)$ be a hypersurface in $\mathbb{S}^{n+1}$ with constant scalar curvature $R=n(n-1)$. If $M$ is locally conformally flat, we have
	$$R_{l i j k}=-\varphi_{i j} \delta_{l k}+\varphi_{i k} \delta_{l j}-\delta_{i j} \varphi_{l k}+\delta_{i k} \varphi_{l j}$$
	where $\varphi_{i j}=\frac{R_{i j}-\frac{n}{2} \delta_{i j}}{n-2}$.
\end{lemma}

\begin{proof}
Recall that the Riemmanian curvature tensor has the following decomposition
$$R_{l i j k}=W_{l i j k}+\frac{1}{n-2}(R_{ik}\delta_{lj}-R_{ij}\delta_{lk}+R_{lj}\delta_{ik}-R_{lk}\delta_{ij})
-\frac{R}{(n-1)(n-2)}(\delta_{ik}\delta_{lj}-\delta_{ij}\delta_{lk}).$$ Since $M$ is locally conformally flat, i.e. $W_{l i j k}=0$, and $R=n(n-1)$, we get $$R_{l i j k}=-\varphi_{i j} \delta_{l k}+\varphi_{i k} \delta_{l j}-\delta_{i j} \varphi_{l k}+\delta_{i k} \varphi_{l j}$$ by direct computation.
\end{proof}

\begin{lemma}[\cite{Ok}]
 Let $a_i , i = 1, 2, . . . , n$ be real numbers satisfying
 $$\sum_{i=1}^{n} a_{i}=0, \sum_{i=1}^{n} a_{i}^{2}=|\mu|^{2},$$
then
$$ \left|\sum_{i=1}^{n} a_{i}^{3}\right| \leq \frac{n-2}{\sqrt{n(n-1)}}|\mu|^{3}.$$
The equality holds if and only if $n-1$ terms of $\left\{a_{i}\right\}_{i=1}^{n}$ are equal.
\end{lemma}
Following the argument in \cite{LXZ}, we have
\begin{lemma}
Let $M^n (n\geq3)$ be a hypersurface  in $\mathbb{S}^{n+1}$ with constant scalar curvature $R=n(n-1)$. Then at least $n-1$ eigenvalues of $h_{i j}$ are zero if
$\sum_jh_{i j} h_{j k}=n H h_{i k}$.
\end{lemma}

\begin{proof}
Since $\sum_j h_{i j} h_{j k}=n H h_{i k}$ and $R=n(n-1)$, by (2.4) we have
$$\rm tr\left(h^{3}\right)=\sum_{i,j,k}h_{i j} h_{j k} h_{i k}=n H \sum_{i, k=1}^{n}\left(h_{i k}\right)^{2}=n^3H^3$$

\noindent Set $\mu_{i j}:=h_{i j}-H \delta_{i j}$. Then we have
$$|\mu|^{2}=(n-1) n H^{2},$$
and
$$
\begin{aligned}
\rm tr\left(h^{3}\right) &=\rm tr\left(\mu^{3}\right)+3 H \rm tr\left(\mu^{2}\right)+n H^{3} \\
&=\rm tr\left(\mu^{3}\right)+3 H|h|^{2}-2 n H^{3}.
\end{aligned}
$$
Therefore, we have
$$\left|\rm tr\left(\mu^{3}\right)\right|=n(n-1)(n-2)\left|H^{3}\right|=\frac{n-2}{\sqrt{n(n-1)}}|\mu|^{3}.$$

\noindent Let $p\in M$ and assume that $\mu_{i j}$ is diagonal at $p$, then by Lemma 2.2 we may assume that $\left(\mu_{i j}\right)=\operatorname{diag}\left\{v_{1}, \ldots, v_{1}, v_{2}\right\}$. Thus at $p$,
$$h_{i j}=\operatorname{diag}\left\{v_{1}+H, \ldots, v_{1}+H, v_{2}+H\right\}$$

\noindent Since $\sum_jh_{i j} h_{j k}=n H h_{i k}$, we have
$$
\left(v_{1}+H\right)^{2}=\left((n-1)\left(v_{1}+H\right)+\left(v_{2}+H\right)\right)\left(v_{1}+H\right),
$$
and
$$
\left(v_{2}+H\right)^{2}=\left((n-1)\left(v_{1}+H\right)+\left(v_{2}+H\right)\right)\left(v_{2}+H\right) .
$$
Therefore  we have $(v_{1} + H) = 0$.

\end{proof}
Recall that the following Michael-Simon's inequality holds true.
\begin{lemma}[\cite{MS}]
	Let $M^n$ be a sub-manifold  in $\mathbb{R}^{n+p}$. Then for any  function $\varphi \in C_{0}^{1}(M)$,
	we have
	$$\left(\int_{M}|\varphi|^{\frac{n}{n-1}} \mathrm{~d} M\right)^{\frac{n-1}{n}} \leq C_{1}\left(\int_{M}|\nabla \varphi| \mathrm{d} M+ \int_{M}|H \varphi| \mathrm{d} M\right)$$
\end{lemma}

Following the argument of \cite{LXZ} and \cite{Xu} (see also \cite{BL}) we get the following variant of Michael-Simon's inequality.

\begin{lemma}
	Let $M^n (n\geq3)$ be a hypersurface in $\mathbb{S}^{n+1}$. Suppose that $\|H\|_{n} C_{1}<1$
	where $C_{1}$ is a constant in Lemma 2.4. Then for any $f \in C_{0}^{1}(M)$ we have
	$$\left(\int_{M}|f|^{\frac{2 n}{n-2}} \mathrm{~d} M\right)^{\frac{n-2}{n}} \leq C_{s} \int_{M}|\nabla f|^{2} +f^2\mathrm{~d} M, \eqno (2.5)$$
	where $C _ { s } = 2( \frac { C _ { 1 } } { 1 - \| H \| _ { n } C _ { 1 } } \frac { 2 n - 2 } { n - 2 } ) ^ { 2 }$
\end{lemma}

\begin{proof}
 Seem $M^n$ has mean curvature $\bar{H}$ as a submanifold in $\mathbb{R}^{n+2}$. Then it is easy to see that
$$|\bar{H}|^{2}=1+|H|^{2}.$$
By Lemma 2.4 we have
$$
\begin{aligned}
\left(\int_{M}|\varphi|^{\frac{n}{n-1}} d M\right)^{\frac{n-1}{n}} &\leq C_{1}\left(\int_{M}|\nabla \varphi|+\int_{M}|\bar{H}||\varphi| d M\right)
\\&\leq C_{1}\left(\int_{M}|\nabla \varphi|+\int_M|\varphi|dM+\int_{M}|H||\varphi| d M\right)
\end{aligned}
$$
Let $\varphi=f^\frac{2(n-1)}{n-2}$. Then by H\"older's inequality, we have
$$\left(\int _ { M } | f | ^ { \frac { 2 n } { n - 2 } } d M )\right)^ { \frac { n - 1 } { n } } \leq C _ { 1 } [\frac { 2 ( n - 1 ) } { n - 2 } ( \int _ { M } | f | ^ { \frac { 2 n } { n - 2 } } d M ) ^ { \frac { 1 } { 2 } } ( \int _ { M } | \nabla f |^ { 2 } d M ) ^ { \frac { 1 } { 2 } }$$
$$+( \int _ { M } f ^ { 2 } d M ) ^ { \frac { 1 } { 2 } } ( \int _ { M } | f | ^ { \frac { 2 n } { n - 2 } } d  M ) ^ { \frac { 1 } { 2 } } +\| H \| ^ { n } ( \int _ { n } | f | ^ { \frac { 2 n } { n - 2 } }d M ) ^ { \frac { n - 1 } { n } }],$$
\noindent and the conclusion follows.
\end{proof}

\section{Proof of Theorem 1.2}

\begin{proof}
\noindent Let $$r _ { i j } = R _ { i j } - ( n - 1 ) \delta _ { i j },$$ then $$r _ { i j , k } = r _ { i k , j },$$ since $M$ is locally conformally flat. By Ricci formula and Lemma 2.1, we have

$$
\begin{aligned}
&\sum_lr _ { i j , l l }
\\&=\sum_lr _ { i l , j l }\\
=&\sum_l r _ { i l , l j } + \sum_{l,k}r _ { k l } R _ { k i j l } +\sum_{l,k} r _ { i k } R _ { k l j l }\\
=& \frac { 1 } { n - 2 } ( 2\sum_k r _ { k i } r _ { k j } + ( n - 2 ) r _ { i j } - \sum_{l,k}r _ { k \ell }^2 \delta _ { i j } ) +\sum_k r _ { i k } r _ { k j } + ( n - 1 )\sum_k r _ { i k } \delta _ { k j }
\\&=\frac{n}{n-2}\sum_{k}r_{i k}r_{kj}-\frac{1}{n-2}|r|^2\delta_{i j}+nr_{i j}.
\end{aligned}
$$

\noindent Multiplying $r _ { i j }$ to both sides of above identity and summing up for $i, j$ from 1 to $n$ we obtain
$$\sum_{i,j}r _ { i j } \Delta r _ { i j } = \frac { n } { n - 2 }\rm tr(r^3) + n | r | ^ { 2 }, \eqno (3.1)$$
where $r$ is the matrix of $(r _ { i j })$ and $|r|^{2}=\sum_{i, j}r_{i j}^{2}$. By
$$\sum_{i,j}r _ { i j } \Delta r _ { i j } + \sum_{i,j}| \nabla r _ { i j } | ^ { 2 } = | r | \Delta | r | + | \nabla | r | |^ { 2 },$$
\noindent we get
$$| r | \Delta | r |\geq \frac { n } { n - 2 }\rm tr(r^3) + n | r | ^ { 2 }, \eqno (3.2)$$

\noindent since $\sum_{i,j}| \nabla r _ { i j } | ^ { 2 }\geq | \nabla | r | |^ { 2 }$.

Let $f$ be a cut-off function supported in a geodesic ball $B(o, R)$ with $o\in M $ such that $$| \nabla f | \leq \frac { C } { R }, | \Delta f | \leq \frac { C } { R ^ { 2 } }.$$ Multiplying both sides of $(3.2)$ by $f ^ { 2 } | r | ^ { q } (q>-1)$, and by integral by parts we have (remark: actually since $|r|$ may have zero points, we need to choose the test function $f ^ { 2 } (| r |+\epsilon)^ { q } (q>-1)$ and then let $\epsilon\to 0^+$)

$$
\begin{aligned}
 0 \geq & 2 \int _ { M } f | r | ^ { q + 1 }\langle \nabla f, \nabla | r |\rangle + ( q + 1 ) \int _ { M } f ^ { 2 } | r | ^ { q }  |\nabla | r | |^ { 2 }\\
& + \frac { n  } { n - 2 } \int _ { M } f ^ { 2 } | r | ^ { q }\rm tr(r^3)+ n\int _ { M }  f ^ { 2 } |r | ^ { q+2 }.
\end{aligned}
$$

\noindent Therefore

$$
\begin{aligned}
n\int _ { M }  f ^ { 2 } |r | ^ { q+2 } + ( q + 1 ) \int _ { M } f ^ { 2 }& | r | ^ { q } |\nabla | r | |^ { 2 }\leq \left| \frac { n  } { n - 2 } \int _ { M } f ^ { 2 } | r | ^ { q }\rm tr(r^3) \right|\\
&+\frac { 2 } { q + 2 } \int _ { M } \rm div( f \nabla f ) | r | ^ { q + 2 }.
\end{aligned}\eqno (3.3)
$$

\noindent Let $q = \frac { n - 4 } { 2 }$. Using the Gauss formula,

$$
\begin{aligned}
\sum_{i,j} r _ { i j } ^ { 2 } &= \sum_{i,j}( n H h _ { i j } -\sum_k h _ { i k } h _ { k j } ) ^ { 2 }\\
&= n ^ { 4 } H ^ { 4 } - 2 n H \rm tr ( h^ { 3 } ) + \sum _ { i , j } ( \sum _ { k } h _ { i k } h _ { k j } ) ^ { 2 } \\
&\leq 4 n ^ { 4 } H ^ { 4 }.
\end{aligned}
$$

\noindent It follows that
$$
\begin{aligned}
&\int _ { M } | r | ^ { q + 2 }= \int _ { M } | r | ^ { n / 2 }
\\&\leq( 4 n ^ { 4 } ) ^ { \frac { n } { 4 } } \int _ { M } | H | ^ { n }
<( 4 n ^ { 4 } ) ^ { \frac { n } { 4 } } \alpha ^ { n } < + \infty.
\end{aligned}
$$
\noindent By H\"older's inequality and Lemma 2.5, we have

$$
\begin{aligned}
\left| \int _ { M } f ^ { 2 } |r | ^ { q } \rm tr(r^3)\right| & \leq  \int _ { M } f ^ { 2 } | r | ^ { q + 3 }\\
& \leq \left( \int _ { M } ( f | r | ^ { \frac { q } { 2 } + 1 } ) ^ { \frac { 2 n } { n - 2 } } \right) ^ { \frac { n - 2 } { n } } \left( \int _ { M } | r | ^ { \frac { n } { 2 } } \right) ^ { \frac { 2 } { n } }\\
&\leq C _ { s } \left( \int _ { M } | r | ^ { \frac { n } { 2 } } \right) ^ { \frac { 2 } { n } } \left( \int _ { M } | \nabla ( f | r | ^ { \frac { q } { 2 } + 1 } ) | ^ { 2 } +f ^ { 2 } | r | ^ { q + 2 }\right)\\
& \leq 2 C _ { s }\left( \int _ { M } | r | ^ { \frac { n } { 2 } } \right) ^ { \frac { 2 } { n } }   [ \int _ { M } | \nabla f | ^ { 2 } | r | ^ { q + 2 }+ ( \frac { q } { 2 } + 1 ) ^ { 2 } \int _ { M } f ^ { 2 } | r | ^ { q } | \nabla | r | |^ { 2 }.\\
& +\int _ { M }f ^ { 2 } | r | ^ { q + 2 })]
\end{aligned}
$$

\noindent Combining  this estimate with (3.3) we get

$$
\begin{aligned}
E\int _ { M }  f ^ { 2 } |r | ^ { q+2 }&+F \int _ { M } f ^ { 2 } | r | ^ { q } | \nabla | r | |^ { 2 } \\
\leq 2C _ { s } \frac { n } { n - 2 } ( \int _ { M } | r | ^ { \frac { n } { 2 } } ) ^ { \frac { 2 } { n } }  \int _ { M } | \nabla& f | ^ { 2 } | r | ^ { q + 2 } + \frac { 2 } { q + 2 } \int _ { M } \rm div( f \nabla f ) | r | ^ { q + 2 },
\end{aligned}
$$

\noindent where $E=n-2C _ { s } (\int _ { M } | r | ^ { \frac { n } { 2 } } ) ^ { \frac { 2 } { n } }$ and $F=q + 1  - 2C _ { s } \frac { n } { n - 2 } ( \frac { q } { 2 } + 1 ) ^ { 2 } ( \int _ { M } | r | ^ { \frac { n } { 2 } } ) ^ { \frac { 2 } { n } }$.
\noindent Assume that $\alpha$ is sufficiently small such that $E, F$ are greater than zero. Letting $R \rightarrow + \infty$, we obtain that $| r | = 0$. Therefore,
$$n Hh_{i j}=\sum_kh_{i k} h_{k j}.$$
By Lemma 2.3, at least $n-1$ eigenvalues of $h_{i j}$ are zero.

At last by following the argument used in the proof of Theorem 2.2 in \cite{Wu},  we can prove that all eigenvalues of  $(h_{i j})$ are zero, i.e. $M$ is totally geodesic. For convenience of the reader we give the details here.
\\
 \\ \textbf{Proof of Proposition 1.1.} Since $n-1$ eigenvalues of $h$ are zero, by the Gauss equation we see that the sectional curvature of $M$ equals to one everywhere. Then by the Bonnet-Myers' theorem we see that $M$ is compact. If $M$ is not totally geodesic, then there is a point $p$ such that the biggest eigenvalue of $h$, say $\lambda_1=\mu$ attains its maximum value $\mu(p)>0$ at $p$. Then there is a neighborhood of $p\in U\subset M$ such that on $U$, $\lambda_1=\mu>0$ and $\lambda_2=...=\lambda_n=0$. Next we restrict our analysis on $U$.

Assume that $2\leq r,s\leq n$, then
 $$
 \begin{aligned}
 \sum_{i}h_{1r,i}\omega_i=dh_{1r}-\sum_k(h_{k1}\omega_{kr}+h_{kr}\omega_{k1})=-\mu\omega_{1r}.
 \end{aligned}
 $$ Note that
 $$
  \begin{aligned}
 \sum_{i}h_{1r,i}\omega_i&=h_{1r,1}\omega_1+\sum_{i>1}h_{1r,i}\omega_i
 \\&=h_{11,r}\omega_1+\sum_{i>1}h_{ir,1}\omega_i
 \\&=\mu_r\omega_1.
 \end{aligned}
 $$

Therefore
 $$\mu_r\omega_1=-\mu\omega_{1r}=\mu\omega_{r1},$$
which implies that
$$\omega_{r1}=\big(\log\mu\big)_r\omega_1.$$
Take the exterior differentiation of the above equation and by the structure equation on $M$ we obtain

$$
\begin{aligned}
d\omega_{r1}&=-\sum_s\omega_{rs}\wedge\omega_{s1}+\omega_r\wedge\omega_1
\\&=-\sum_s\big(\log\mu\big)_s\omega_{rs}\wedge\omega_1+\omega_r\wedge\omega_1,
\\d\omega_{r1}&=d(\big(\log\mu\big)_r\omega_1)=-\big(\log\mu\big)_r\sum_s\omega_{1s}\wedge\omega_s
+\sum_s[(\log\mu)_{rs}\omega_s+\big(\log\mu\big)_s\omega_{sr}]\wedge\omega_1
\\&=\sum_s-\big(\log\mu\big)_r\big(\log\mu\big)_s\omega_s\wedge\omega_1+\sum_s\big(\log\mu\big)_{rs}\omega_s\wedge\omega_1-\sum_s\big(\log\mu\big)_s\omega_{rs}\wedge\omega_1.
\end{aligned}
$$

Comparing the above formulas we get on $U$
$$\big(\log\mu\big)_{rs}=\big(\log\mu\big)_r\big(\log\mu\big)_s+\delta_{rs}.$$
Since $p$ is a maximum point of $\mu$ on $M$, we see that at $p$,
$$0\geq\big(\log\mu\big)_{rr}=\big(\log\mu\big)_r^2+\delta_{rr}=1,$$
a contradiction. Hence we get the conclusion of Proposition 1.1.

This completes the proof of Theorem 1.2. \end{proof}
\vspace{0.5cm}
\textbf{Declarations}
\\

\textbf{Funding} The second author is supported by the NSF of China (Grant No. 12271069) and Chongqing NSF (Grant No. cstc2021jcjy-msxmX0443).
\\

\textbf{Conflict of interest} The authors hereby state that there are no conflicts of
interest regarding the presented results.
\\

\textbf{Data availability statements} No data sets were generated or analysed during the current study.

\vspace{1cm}\sc

Jinchuan Bai, Yong Luo

Mathematical Science Research Center of Mathematics,

Chongqing University of Technology,

Chongqing, 400054, China

{\tt baijinchuan23@163.com, yongluo-math@cqut.edu.cn}

\end{document}